\documentclass[11 pt]{amsart}
\usepackage{latexsym}
\usepackage{amssymb}
\usepackage{amsthm}
\usepackage{tikz}
\usepackage{amscd}
\usepackage[all,cmtip]{xy}
\makeatletter
\@namedef{subjclassname@2010}{\textup{2010} Mathematics Subject Classification}
\makeatother

\newtheorem{theorem}{Theorem}[section]
\newtheorem{lemma}[theorem]{Lemma}

\newtheorem{corollary}[theorem]{Corollary}
\theoremstyle{definition}

\theoremstyle{remark}

\newtheorem*{acknowledgement}{Acknowledgement}

\begin{document}
\title{A note on the commutativity of inverse limit and orbit map}
\author{Mahender Singh}
\address{Institute of Mathematical Sciences, C I T Campus, Taramani, Chennai 600113, India.}
\email{mahen51@gmail.com}

\begin{abstract}
We show that the inverse limit and the orbit map commute for actions of compact groups on compact Hausdorff spaces.
\end{abstract}
\maketitle

\section{Introduction}
This note is motivated by the following example of Bredon \cite[p.145]{Bredon2}. Let $\mathbb{S}^2$ be the 2-sphere identified with the unreduced suspension of the circle $\mathbb{S}^1= \{z \in \mathbb{C}~;~ |z|=1 \}$, and let $f:\mathbb{S}^2 \to \mathbb{S}^2$ be the suspension of the map $z \mapsto z^3$ from $\mathbb{S}^1 \to \mathbb{S}^1$. Then $f$ commutes with the antipodal involution on $\mathbb{S}^2$. If $\Sigma$ is the inverse limit of the inverse system $$\cdots \stackrel{f}\rightarrow \mathbb{S}^{2} \stackrel{f}\rightarrow \mathbb{S}^{2} \stackrel{f}\rightarrow \mathbb{S}^{2},$$ then $\Sigma/ \mathbb{Z}_2$ is homeomorphic to $\varprojlim \mathbb{R}P^2$.

We show that this is also true in a more general setting. More precisely, we show that the inverse limit and the orbit map commute for actions of compact groups on compact Hausdorff spaces. The proof of the result is simple, but does not seem to be available in the literature. Before we prove the result, we recall some basic definitions that will be used in the note.

An inverse system of topological spaces, denoted by $\{X_{\alpha}, \pi_{\alpha}^{ \beta},\Lambda\}$, consists of a directed set $\Lambda$, a family of topological spaces $\{X_{\alpha}\}_{\alpha \in \Lambda}$, and a collection of continuous maps $\pi_{\alpha}^{\beta}:X_{\beta} \to X_{\alpha}$ for $\alpha < \beta $ in $\Lambda$ satisfying the following:
\begin{enumerate}
\item $\pi_{\alpha}^{ \alpha}:X_{\alpha} \to X_{\alpha} $ is the identity map for all $\alpha\; \in \Lambda$;
\item $\pi_{\alpha}^{ \beta} \pi_{\beta}^{ \gamma}\;=\; \pi_{\alpha}^{ \gamma}$ for all $\alpha < \beta< \gamma$ in $\Lambda$.
\end{enumerate}
The maps $\pi_{\alpha}^{ \beta}$ are called bonding maps. Given an inverse system $ \{X_{\alpha}, \pi_{\alpha}^{ \beta},\Lambda \}$ of topological spaces, let $\varprojlim X_{\alpha}$ (possibly empty) be the subset of $\Pi_{\alpha \in \Lambda} X_{\alpha}$ consisting of elements $(x_{\alpha})$ such that $x_{\alpha}=\pi_{\alpha}^{ \beta}(x_{\beta})$ for $\alpha < \beta $. It is given the subspace topology from the product space $\Pi_{\alpha \in \Lambda} X_{\alpha}$ and is called the inverse limit of the inverse system $\{X_{\alpha}, \pi_{\alpha}^{ \beta}, \Lambda \}$. We denote by $\pi_{\beta}: \varprojlim X_{\alpha} \to X_{\beta}$, the restriction of the canonical projection $\Pi_{\alpha \in \Lambda} X_{\alpha}\to X_{\beta}$. The following results are well known and we refer the reader to \cite{Eilenberg} for more details.

\begin{theorem}\cite[p.217]{Eilenberg}\label{theorem1.1}
The inverse limit of an inverse system of non-empty compact Hausdorff topological spaces is a non-empty compact Hausdorff topological space.
\end{theorem}

\begin{theorem}\cite[p.219]{Eilenberg}\label{theorem1.2}
Let $X$ be a topological space and $\{X_{\alpha}, \pi_{\alpha}^{ \beta},\Lambda\}$ be an inverse system of topological spaces. If for each $\alpha \in \Lambda$ there is a continuous map $\psi_{\alpha}:X \to X_{\alpha}$ such that $\pi_{\alpha}^{\beta} \psi_{\beta}=\psi_{\alpha}$ for $\alpha < \beta$, then there is a unique continuous map $\psi:X \to \varprojlim X_{\alpha}$.
\end{theorem}

If $\{G_{\alpha}, \nu_{\alpha}^{ \beta},\Lambda\}$ is an inverse system of topological groups ($\nu_{\alpha}^{ \beta}$ are continuous group homomorphisms), then $\varprojlim G_{\alpha}$ is also a topological group with the subspace topology from $\Pi_{\alpha \in \Lambda} G_{\alpha}$ and with the group operation $\big((g_{\alpha}), (h_{\alpha})\big) \mapsto (g_{\alpha} h_{\alpha})$.

An action of a topological group $G$ on a topological space $X$ is a continuous map $G \times X \to X$, written $(g, x)\mapsto gx$ for $g \in G$ and $x \in X$, and satisfying the following:
\begin{enumerate}
\item $ex=x$ for all $x \in X$, where $e \in G$ is the identity element;
\item $(gh)x= g(hx)$ for all $g,~h \in G$ and $x \in X$.
\end{enumerate}

We say that $X$ is a $G$-space if there is an action of $G$ on $X$. For each $x \in X$, the set $\overline{x}= \{ gx~|~ g \in G\}$ is called the orbit of $x$. If $X/G$ is the set of all orbits, then the canonical map $X \to X/G$ given by $x \mapsto \overline{x}$ is called the orbit map and $X/G$ equipped with the quotient topology is called the orbit space. We say that $G$ acts freely on $X$ if for each $x \in X$, $gx=x$ implies that $g=e$. The following result is well known.

\begin{theorem}\cite[p.38]{Bredon2}\label{theorem1.3}
If $G$ is a compact topological group acting on a compact Hausdorff topological space $X$, then $X/G$ is also a compact Hausdorff topological space.
\end{theorem}

Let $X$ be a $G$-space, $Y$ be a $H$-space and $\nu:G \to H$ be a topological group homomorphism. Then a continuous map $f:X \to Y$ is called $\nu$-equivariant if $$f(g x)=\nu(g)f(x) ~\textrm{for~ all}~ g\in G ~\textrm{and}~ x \in X.$$ If both $X$ and $Y$ are $G$-spaces, then $f$ is simply called $G$-equivariant. Note that the $\nu$-equivariant map $f$ induces a continuous map $\overline{f}:X/G \to Y/H$ given by $\overline{f}(\overline{x})=\overline{f(x)}$.

If $\{X_{\alpha}, \pi_{\alpha }^{\beta}, \Lambda \}$ is an inverse system of topological spaces and $\{G_{\alpha}, \nu_{\alpha}^{ \beta},\Lambda\}$ is an inverse system of topological groups, where each $X_{\alpha}$ is a $G_{\alpha}$-space and each bonding map $\pi_{\alpha }^{\beta}$ is $\nu_{\alpha }^{\beta}$-equivariant, then we get another inverse system of topological spaces $\{X_{\alpha}/G_{\alpha}, \overline{\pi_{\alpha}^{ \beta}}, \Lambda \}$ by passing to orbit spaces. Also, under above conditions $\varprojlim X_{\alpha}$ is a $\varprojlim G_{\alpha}$-space with the action given by $$(g_{\alpha}) (x_{\alpha})=(g_{\alpha} x_{\alpha})~~\textrm{for}~ (g_{\alpha}) \in \varprojlim G_{\alpha}~\textrm{and}~ (x_{\alpha})\in \varprojlim X_{\alpha}.$$

\section{Commutativity of inverse limit and orbit map}
In view of the above discussion, it is natural to ask when is $(\varprojlim X_{\alpha})/(\varprojlim G_{\alpha})$ homeomorphic to $\varprojlim (X_{\alpha}/G_{\alpha})$. We present the following theorem in this direction.

\begin{theorem}\label{theorem2.1}
Let $\{X_{\alpha}, \pi_{\alpha}^{ \beta}, \Lambda \}$ be an inverse system of non-empty compact Hausdorff topological spaces and let $\{G_{\alpha}, \nu_{\alpha}^{ \beta},\Lambda\}$ be an inverse system of compact topological groups, where each $X_{\alpha}$ is a $G_{\alpha}$-space and each bonding map $\pi_{\alpha }^{\beta}$ is $\nu_{\alpha }^{\beta}$-equivariant. Further, assume that $\Lambda$ has the least element $\lambda$, $G_{\lambda}$ action on $X_{\lambda}$ is free and the bonding map $\nu_{\lambda}^{\alpha}$ is injective for each $\alpha \in \Lambda$. Then there is a homeomorphism $$\psi :(\varprojlim X_{\alpha})/(\varprojlim G_{\alpha}) \to \varprojlim (X_{\alpha}/G_{\alpha}).$$
\end{theorem}

First we prove the following simple lemma.
\begin{lemma}\label{lemma2.2}
Let $ \{X_{\alpha}, \pi_{\alpha}^{ \beta}, \Lambda \}$ be an inverse system of non-empty compact Hausdorff topological spaces and let $ \{G_{\alpha}, \nu_{\alpha}^{ \beta},\Lambda\}$ be an inverse system of compact topological groups, where each $X_{\alpha}$ is a $G_{\alpha}$ -space and each bonding map $\pi_{\alpha }^{\beta}$ is $\nu_{\alpha }^{\beta}$-equivariant. Then there is a closed continuous surjection $$\psi :(\varprojlim X_{\alpha})/(\varprojlim G_{\alpha}) \to \varprojlim (X_{\alpha}/G_{\alpha}).$$
\end{lemma}

\begin{proof}
Let $X=\varprojlim X_{\alpha}$ and $G=\varprojlim G_{\alpha}$. Let $\pi_{\beta}:X \to X_{\beta}$ and $\nu_{\beta}:G \to G_{\beta}$ be the canonical projections for each $\beta \in \Lambda$. Clearly each $\pi_{\beta}$ is $\nu_{\beta}$-equivariant and therefore induces a continuous map $\psi_{\beta}:X/G \to X_{\beta}/G_{\beta}$ given by $\overline{(x_{\alpha})} \mapsto \overline{x_{\beta}}$ (note that $\psi_{\beta}= \overline{\pi_{\beta}}$). Also, observe that for $\gamma < \beta$, the diagram
$$
\xymatrix{
X/G \ar[d]^{\psi_{\gamma}} \ar[r]^{\psi_{\beta}} & X_{\beta}/G_{\beta} \ar[ld]^{\overline{\pi_{\gamma}^{ \beta}}} \\
X_{\gamma}/G_{\gamma}}
$$
commutes. Therefore, by Theorem \ref{theorem1.2}, we have the continuous map $$\psi:X/G \to \varprojlim (X_{\alpha}/G_{\alpha})$$ given by $\overline{(x_{\alpha})} \mapsto (\overline{x_{\alpha}})$. Clearly $\psi$ is surjective. By Theorem \ref{theorem1.1} and Theorem \ref{theorem1.3}, we see that $X/G$ is compact and $\varprojlim (X_{\alpha}/G_{\alpha})$ is Hausdorff. Therefore $\psi$ is a closed map. This completes the proof of the lemma.
\end{proof}

We now complete the proof of Theorem \ref{theorem2.1}.
\begin{proof}
It just remains to show that the map $\psi$ is injective. Suppose $(\overline{x_{\alpha}})=(\overline{y_{\alpha}})$. This means $\overline{x_{\alpha}}=\overline{y_{\alpha}}$ for each $\alpha \in \Lambda$. Thus for $\alpha \in \Lambda$, we have $x_{\alpha}=g_{\alpha} y_{\alpha}$ and $x_{\lambda}=g_{\lambda}  y_{\lambda}$ for some $g_{\alpha} \in G_{\alpha}$ and $g_{\lambda} \in G_{\lambda}$. Since $\lambda \in \Lambda$ is the least element, we have $\lambda < \alpha$ and hence $\pi_{\lambda}^{\alpha}(x_{\alpha})=x_{\lambda}$ and  $\pi_{\lambda}^{\alpha}(y_{\alpha})=y_{\lambda}$. This gives $x_{\lambda}=\pi_{\lambda}^{\alpha}(x_{\alpha})=  \pi_{\lambda}^{\alpha}(g_{\alpha}  y_{\alpha})= \nu_{\lambda}^{\alpha}(g_{\alpha}) \pi_{\lambda}^{\alpha}(y_{\alpha})=\nu_{\lambda}^{\alpha}(g_{\alpha}) y_{\lambda}$ and hence $g_{\lambda} y_{\lambda}=\nu_{\lambda}^{\alpha}(g_{\alpha}) y_{\lambda}$. The freeness of $G_{\lambda}$ action on $X_{\lambda}$ implies $g_{\lambda}=\nu_{\lambda}^{\alpha}(g_{\alpha})$. Now, for $\lambda <\alpha < \beta$ we have $\nu_{\lambda}^{\alpha} \nu_{\alpha}^{\beta}(g_{\beta})=\nu_{\lambda}^{\beta}(g_{\beta})=g_{\lambda}$. By injectivity of $\nu_{\lambda}^{\alpha}$ we get $\nu_{\alpha}^{\beta}(g_{\beta})=g_{\alpha}$. Thus $(g_{\alpha}) \in G$ and $(x_{\alpha})=(g_{\alpha}  y_{\alpha})=(g_{\alpha})  (y_{\alpha})$. This shows that $\psi$ is injective and hence a homeomorphism by Lemma \ref{lemma2.2}.
\end{proof}

As a consequence we have the following corollary.
\begin{corollary}
Let $G$ be a compact topological group acting freely on a compact Hausdorff topological space $X$ and let $f:X \to X$ be a $G$-equivariant map. Then for the inverse system $$\cdots \stackrel{f}\rightarrow X \stackrel{f}\rightarrow X \stackrel{f}\rightarrow X$$ $(\varprojlim X)/G$ is homeomorphic to $\varprojlim(X/G)$.
\end{corollary}
\bigskip

\begin{acknowledgement}
The author would like to thank the referees for comments which improved the presentation of the note.
\end{acknowledgement}
\bigskip

\bibliographystyle{amsplain}

\end{document}